\documentclass[8pt,oneside]{amsart}
\usepackage[T1]{fontenc}
\usepackage{mathtools}
\usepackage{bm}
\usepackage{amstext}
\usepackage{amsthm}
\usepackage{amssymb}

\makeatletter
\numberwithin{equation}{section}
\numberwithin{figure}{section}
\theoremstyle{plain}
\newtheorem{thm}{\protect\theoremname}
\theoremstyle{definition}
\newtheorem{example}[thm]{\protect\examplename}
\theoremstyle{plain}
\newtheorem{conjecture}[thm]{\protect\conjecturename}
\theoremstyle{plain}
\newtheorem{lem}[thm]{\protect\lemmaname}
\theoremstyle{remark}
\newtheorem{rem}[thm]{\protect\remarkname}

\usepackage{listings}
\usepackage{shuffle}

\makeatother

\providecommand{\conjecturename}{Conjecture}
\providecommand{\examplename}{Example}
\providecommand{\lemmaname}{Lemma}
\providecommand{\remarkname}{Remark}
\providecommand{\theoremname}{Theorem}

\begin{document}
\address[Minoru Hirose]{Institute for Advanced Research, Nagoya University,  Furo-cho, Chikusa-ku, Nagoya, 464-8602, Japan}
\email{minoru.hirose@math.nagoya-u.ac.jp}
\title{Multitangent functions and symmetric multiple zeta values}
\author{Minoru Hirose}
\begin{abstract}
In this paper, we give a formula that connects two variants of multiple
zeta values; multitangent functions and symmetric multiple zeta values.
As an application of this formula, we give two results. First, we
prove Bouillot's conjecture on the structures of the algebra of multitangent
functions. Second, we prove an analogue of the linear part of Kawashima's
relation for symmetric multiple zeta values.
\end{abstract}

\maketitle

\section{Introduction}

Multiple zeta values are real numbers defined by 
\[
\zeta(k_{1},\dots,k_{r})=\sum_{0<m_{1}<\cdots<m_{r}}\frac{1}{m_{1}^{k_{1}}\cdots m_{r}^{k_{r}}}
\]
for positive integers $k_{1},\dots,k_{r}$ with $k_{r}\geq2$. The
purpose of this paper is to find a new relationship between two variants
of multiple zeta values (multitangent functions and symmetric multiple
zeta values) that seem to have been studied in slightly different
contexts.

\subsection{Multitangent functions}

Multitangent functions are functions of $z\in\mathbb{C}\setminus\mathbb{Z}$
defined by the absolutely convergent series
\[
\Psi_{k_{1},\dots,k_{d}}(z)=\sum_{-\infty<m_{1}<\cdots<m_{d}<\infty}\frac{1}{(z+m_{1})^{k_{1}}\cdots(z+m_{d})^{k_{d}}}
\]
where $k_{1},\dots,k_{r}$ are positive integers with $k_{1},k_{d}\geq2$.
The multitangent functions appear in the theory of resurgent functions
\cite{Bou11InvAna,BouEca12,EcaLFR2} and the theory of multiple Eisenstein
series \cite{Bach15Thesis}. In \cite{Bou14AlgMTGF}, Bouillot studied
the algebraic structure of multiple tangent functions and obtained
the following results. We extend $\Psi_{\bm{k}}(z)$ to the case where
$\bm{k}$ is a formal sum of indices whose first and last components
are greater than $1$. Furthermore, let $\bm{k}*\bm{l}$ denote the
harmonic product of $\bm{k}$ and $\bm{l}$, e.g. $(a)*(b)=(a,b)+(b,a)+(a+b)$.
We call the $d=1$ case $\Psi_{k}(z)=\sum_{m=-\infty}^{\infty}(z+m)^{-k}$
as a monotangent function. For the empty index $\emptyset$, we put
$\zeta(\emptyset)=1$.
\begin{thm}[{\cite[Property 3, Theorem 3]{Bou14AlgMTGF} }]
\label{thm:Bou}
\end{thm}

\begin{description}
\item [{(i)}] For $\bm{k}=(k_{1},\dots,k_{r})\in\mathbb{Z}_{\geq1}^{r}$
with $k_{1},k_{r}\geq2$, the function $\Psi_{\bm{k}}(z)$ is well-defined
and holomorphic on $\mathbb{C}\setminus\mathbb{Z}$.
\item [{(ii)}] For $\bm{k}=(k_{1},\dots,k_{r})\in\mathbb{Z}_{\geq1}$ with
$k_{1},k_{r}\geq2$ and $z\in\mathbb{C}\setminus\mathbb{Z}$, we have
\[
\frac{\partial\Psi_{\bm{k}}(z)}{\partial z}=-\sum_{i=1}^{r}k_{i}\Psi_{k_{1},\dots,k_{i-1},k_{i}+1,k_{i+1},\dots,k_{r}}(z).
\]
\item [{(iii)}] For $\bm{k}=(k_{1},\dots,k_{r})\in\mathbb{Z}_{\geq1}$
with $k_{1},k_{r}\geq2$ and $z\in\mathbb{C}\setminus\mathbb{Z}$,
we have
\[
\Psi_{\bm{k}}(-z)=(-1)^{k_{1}+\cdots+k_{r}}\Psi_{k_{r},\dots,k_{1}}(z).
\]
\item [{(iv)}] For $\bm{k}=(k_{1},\dots,k_{r})\in\mathbb{Z}_{\geq1}^{r}$
and $\bm{l}=(l_{1},\dots,l_{s})\in\mathbb{Z}_{\geq1}^{s}$ with $k_{1},k_{r},l_{1},l_{s}\geq2$
and $z\in\mathbb{C}\setminus\mathbb{Z}$, we have
\[
\Psi_{\bm{k}*\bm{l}}(z)=\Psi_{\bm{k}}(z)\Psi_{\bm{l}}(z).
\]
\item [{(v)}] For $\bm{k}=(k_{1},\dots,k_{r})\in\mathbb{Z}_{\geq1}^{r}$
with $k_{1},k_{r}\geq2$ and $z\in\mathbb{C}\setminus\mathbb{Z}$,
we have
\[
\Psi_{{\bf k}}(z)=\sum_{j=1}^{d}\sum_{\substack{a+s+b=k_{j}\\
a,b\geq0,\,s\geq2
}
}(-1)^{k_{1}+\cdots+k_{j-1}+a}\zeta_{a}(k_{j-1},\dots,k_{1})\zeta_{b}(k_{j+1},\dots,k_{d})\Psi_{s}(z)
\]
where we put
\begin{align*}
\zeta_{a}(k_{1},\dots,k_{d}) & =(-1)^{a}\sum_{\substack{e_{1}+\cdots+e_{d}=a\\
e_{1},\dots,e_{d}\geq0
}
}\zeta(k_{1}+e_{1},\dots,k_{d}+e_{d})\prod_{j=1}^{d}{k_{j}-1+e_{j} \choose e_{j}}.
\end{align*}
Thus, any multitangent function can be written as a $\mathbb{Q}$-linear
sum of products of multiple zeta values and monotangent functions.
\end{description}
\begin{example}
The case $\bm{k}=(2)$, $\bm{l}=(3)$ of (iv) implies 
\[
\Psi_{2}(z)\Psi_{3}(z)=\Psi_{2,3}(z)+\Psi_{3,2}(z)+\Psi_{5}(z).
\]
\end{example}

\begin{example}
The case $\bm{k}=(2,3)$ of (v) implies
\[
\Psi_{2,3}(z)=\sum_{\substack{a+s+b=2\\
a,b\geq0,\,s\geq2
}
}(-1)^{a}\zeta_{a}(\emptyset)\zeta_{b}(3)\Psi_{s}(z)+\sum_{\substack{a+s+b=3\\
a,b\geq0,\,s\geq2
}
}(-1)^{a+2}\zeta_{a}(2)\zeta_{b}(\emptyset)\Psi_{s}(z).
\]
Since $\zeta_{a}(\emptyset)=\delta_{a,0}$, the right hand side is
equal to
\[
\zeta_{0}(3)\Psi_{2}(z)-\zeta_{1}(2)\Psi_{2}(z)+\zeta_{0}(2)\Psi_{3}(z)=3\zeta(3)\Psi_{2}(z)+\zeta(2)\Psi_{3}(z).
\]
Thus,
\[
\Psi_{2,3}(z)=3\zeta(3)\Psi_{2}(z)+\zeta(2)\Psi_{3}(z).
\]
\end{example}

For an index $\bm{k}=(k_{1},\dots,k_{r})$, we call the sum $k_{1}+\cdots+k_{r}$
as the weight. We denote by $\mathsf{MTGF}_{w}$ (resp. $\mathsf{MZV}_{w}$)
the $\mathbb{Q}$-vector space generated by multitangent functions
(resp. multiple zeta values) of weight $w$, i.e.,
\[
\mathsf{MTGF}_{w}=\mathrm{span}_{\mathbb{Q}}\left\{ \Psi_{k_{1},\dots,k_{r}}(z)\left|\begin{array}{c}
r\geq1,\,k_{1},\dots,k_{r}\geq1,\\
k_{1},k_{r}\geq2,\,k_{1}+\cdots+k_{r}=w
\end{array}\right.\right\} \subset\left\{ \begin{array}{c}
\text{holomorphic functions }\\
\text{on }\mathbb{C}\setminus\mathbb{Z}
\end{array}\right\} ,
\]
\[
\mathsf{MZV}_{w}=\mathrm{span}_{\mathbb{Q}}\left\{ \zeta(k_{1},\dots,k_{r})\left|\begin{array}{c}
r\geq0,\,k_{1},\dots,k_{r}\geq1,\\
k_{r}\geq2\text{ if }r>0,\ k_{1}+\cdots+k_{r}=w
\end{array}\right.\right\} \subset\mathbb{R}.
\]
Then Theorem \ref{thm:Bou} (v) implies
\[
\mathsf{MTGF}_{w}\subset\bigoplus_{s=2}^{w}\mathsf{MZV}_{w-s}\cdot\Psi_{s}(z).
\]
Bouillot \cite{Bou14AlgMTGF} conjectured that this inclusion is actually
an equality.
\begin{conjecture}[{Bouillot, \cite[Conjecture 2]{Bou14AlgMTGF}}]
\label{conj:Bou}For $w\geq2$, we have
\[
\mathsf{MTGF}_{w}=\bigoplus_{s=2}^{w}\mathsf{MZV}_{w-s}\cdot\Psi_{s}(z).
\]
\end{conjecture}

Furthermore, Bouillot gave several equivalent forms of the conjectures.
Put $\mathsf{MZV}=\bigoplus_{w=0}^{\infty}\mathsf{MZV}_{w}$ and $\mathsf{MTGF}=\bigoplus_{w=2}^{\infty}\mathsf{MTGF}_{w}$.
\begin{thm}[{Bouillot, \cite[Property 4 and Property 5]{Bou14AlgMTGF}}]
\label{thm:Bou-equiv}The following are equivalent.
\begin{description}
\item [{(a)}] Conjecture 4.
\item [{(b)}] $\mathsf{MTGF}$ is a graded $\mathsf{MZV}$-module, i.e.,
$\mathsf{MZV}_{p}\cdot\mathsf{MTGF}_{q}\subset\mathsf{MTGF}_{p+q}$
for $p\geq0$, $q\geq2$, or more explicitly,
\[
\zeta(k_{1},\dots,k_{r})\Psi_{l_{1}+\cdots+l_{s}}(z)\in\mathsf{MTGF}_{k_{1}+\cdots+k_{r}+l_{1}+\cdots+l_{s}}
\]
for $(k_{1},\dots,k_{r})\in\mathbb{Z}_{\geq1}^{r}$ and $(l_{1},\dots,l_{s})\in\mathbb{Z}_{\geq1}^{s}$
with $k_{r},l_{1},l_{s}\geq2$.
\item [{(c)}] For $(k_{1},\dots,k_{r})\in\mathbb{Z}_{\geq1}^{r}$ with
$k_{r}\geq2$, we have
\[
\zeta(k_{1},\dots,k_{r})\Psi_{2}(z)\in\mathsf{MTGF}_{k_{1}+\cdots+k_{r}+2}.
\]
\end{description}
\end{thm}

\subsection{Symmetric multiple zeta values}

Symmetric multiple zeta values are variants of multiple zeta values
introduced by Kaneko and Zagier \cite{KZ} as the real counterpart
of finite multiple zeta values
\[
\zeta^{\mathcal{A}}(l_{1},\dots,l_{r}):=\left(\sum_{0<m_{1}<\cdots<m_{r}}\frac{1}{m_{1}^{l_{1}}\cdots m_{r}^{l_{r}}}\bmod p\right)_{p}\in\left(\prod_{p:\mathrm{prime}}\mathbb{Z}/p\mathbb{Z}\right)/\left(\bigoplus_{p:\mathrm{prime}}\mathbb{Z}/p\mathbb{Z}\right).
\]
Symmetric multiple zeta values are defined as follows. For $\bm{l}=(l_{1},\dots,l_{r})\in\mathbb{Z}_{\geq1}^{r}$,
we denote by $\zeta_{\shuffle}(\bm{l};T)\in\mathbb{R}[T]$ (resp.
$\zeta_{*}(\bm{l};T)\in\mathbb{R}[T]$) the shuffle (resp. harmonic)
regularized multiple zeta polynomial (see \cite{IKZ}\footnote{In \cite{IKZ}, the notation $Z_{\bm{l}}^{\bullet}(T)$ is used for
$\zeta_{\bullet}(\bm{l};T)$ where $\bullet\in\{\shuffle,*\}$.}). They are characterized by the properties $\zeta_{\bullet}(\bm{l};T)=\zeta(\bm{l})$
for $l_{r}\geq2$, $\zeta_{\bullet}(1;T)=T$, and the product formula
$\zeta_{\bullet}(\bm{k}\bullet\bm{l};T)=\zeta_{\bullet}(\bm{k};T)\zeta_{\bullet}(\bm{l};T)$
where $\bullet\in\{\shuffle,*\}$. For $\bm{l}=(l_{1},\dots,l_{r})\in\mathbb{Z}_{\geq1}^{r}$,
we define the shuffle ($\bullet=\shuffle$) or harmonic ($\bullet=*$)
regularized symmetric multiple zeta polynomials by
\[
\zeta_{\bullet}^{S}(l_{1},\dots,l_{r};T)\coloneqq\sum_{i=0}^{r}(-1)^{l_{i+1}+\cdots+l_{r}}\zeta_{\bullet}(l_{1},\dots,l_{i};0)\zeta_{\bullet}(l_{r},\dots,l_{i+1};T)\in\mathbb{R}[T].
\]
By definition,
\[
\zeta_{\bullet}^{S}(\bm{l};0)\in\mathsf{MZV}_{\mathrm{wt}(\bm{l})}
\]
where $\mathrm{wt}(\bm{l})=l_{1}+\cdots+l_{r}$. It is known that
\[
\zeta_{\shuffle}^{S}(\bm{l};0)\equiv\zeta_{*}^{S}(\bm{l};0)\pmod{\zeta(2)\mathsf{MZV}_{\mathrm{wt}(\bm{l})-2}}.
\]
Then, for $\bm{l}=(l_{1},\dots,l_{r})\in\mathbb{Z}_{\geq1}^{r}$,
the symmetric multiple zeta value $\zeta^{S}(\bm{l})\in\mathsf{MZV}_{\mathrm{wt}(\bm{l})}/\zeta(2)\mathsf{MZV}_{\mathrm{wt}(\bm{l})-2}$
is defined by
\[
\zeta^{S}(\bm{l}):=(\zeta_{\bullet}^{S}(\bm{l};0)\bmod\{\zeta(2)\mathsf{MZV}_{\mathrm{wt}(\bm{l})-2}\})
\]
which does not depend on the choice of $\bullet\in\{\shuffle,*\}$.
It is conjectured that the set of $\mathbb{Q}$-linear relations among
$\zeta^{S}(\cdots)$ coincides with the one of $\zeta^{\mathcal{A}}(\cdots)$.
Yasuda showed that the symmetric multiple zeta values generate the
whole space of multiple zeta values.
\begin{thm}[Yasuda, \cite{Yas}]
\label{thm:smzv-span-mzv}For $l\ge0$, we have
\[
\mathsf{MZV}_{l}/\zeta(2)\mathsf{MZV}_{l-2}=\mathrm{span}_{\mathbb{Q}}\{\zeta^{S}(\bm{l})\mid\mathrm{wt}(\bm{l})=l\}.
\]
\end{thm}

In \cite{Hir20Ref}, the author defined refined symmetric multiple
zeta values $\zeta^{RS}(l_{1},\dots,l_{r})$, which are complex numbers
defined by the iterated integral along a non-trivial simple closed
path from $0$ to $0$ on $\mathbb{P}^{1}\setminus\{0,1,\infty\}$.
They can be expressed by both of shuffle and harmonic regularized
symmetric multiple zeta polynomials as
\begin{equation}
\zeta^{RS}(\bm{l})=\frac{1}{2\pi i}\int_{0}^{2\pi i}\zeta_{\shuffle}^{S}(\bm{l};T)dT=\zeta_{*}^{S}(\bm{l};\pi i)\label{eq:zeta_RS}
\end{equation}
(see \cite[Corollary 12]{Hir20Ref}).

\subsection{Main results}

For $\bm{k}=(k_{1},\dots,k_{r})\in\mathbb{Z}_{\geq1}^{r}$ with $r>0$,
Hoffman's dual index $\bm{k}^{\vee}$ is defined by
\[
\bm{k}^{\vee}=(\underbrace{1,\dots,1}_{k_{1}}+\underbrace{1,\dots,1}_{k_{2}}+1,\dots,1+\underbrace{1,\dots,1}_{k_{r}}).
\]
For example,
\[
(2,1,1,3,4)^{\vee}=(\underbrace{1,1}_{2}+\underbrace{1}_{1}+\underbrace{1}_{1}+\underbrace{1,1,1}_{3}+\underbrace{1,1,1,1}_{4})=(1,4,1,2,1,1,1).
\]
For $\bm{k}=(k_{1},\dots,k_{r})\in\mathbb{Z}_{\ge1}^{r}$ with $k_{1}\geq1$
(resp. $k_{r}\geq1$), the symbol $_{\downarrow}\bm{k}$ (resp. $\bm{k}_{\downarrow}$)
denotes the index obtained by decrementing the first (resp. last)
component by $1$. Furthermore, we simply write $_{\downarrow}\bm{k}_{\downarrow}$
for $_{\downarrow}(\bm{k}_{\downarrow})=({}_{\downarrow}\bm{k})_{\downarrow}$,
i.e., 
\[
_{\downarrow}(k_{1},\dots,k_{r})_{\downarrow}=(k_{1}-1,k_{2},\dots,k_{r-1},k_{r}-1).
\]
The first main result of this paper is the following formula that
connects the multitangent functions and the symmetric multiple zeta
values. 
\begin{thm}
\label{thm:main}For $\bm{k}=(k_{1},\dots,k_{d})\in\mathbb{Z}_{\geq1}^{d}$
with $k_{1},k_{d}\geq2$, we have
\[
\Psi_{\bm{k}}(z)=(-1)^{k_{1}+\cdots+k_{d}}\rho_{z}(\zeta_{\shuffle}^{S}(({}_{\downarrow}\bm{k}_{\downarrow})^{\vee};T))
\]
where $\rho_{z}:\mathbb{R}[T]\to\bigoplus_{m=2}^{\infty}\mathbb{R}\Psi_{m}(z)$
is the $\mathbb{R}$-linear map defined by
\[
\rho_{z}(T^{s})=s!\Psi_{s+2}(z).
\]
\end{thm}

Furthermore, we get the following results as applications of Theorem
\ref{thm:main}.
\begin{thm}
\label{thm:main-conjBou}Conjecture \ref{conj:Bou} is true.
\end{thm}

\begin{thm}[Analogue of Kawashima relation for symmetric multiple zeta values]
\label{thm:main-Kawashima}For an index $\bm{k}=(k_{1},\dots,k_{r})\in\mathbb{Z}_{\geq1}^{r}$
and $\bm{l}=(l_{1},\dots,l_{s})\in\mathbb{Z}_{\geq1}^{s}$ with $k_{1},k_{r},l_{1},l_{s}\geq2$,
we have
\[
\zeta^{RS}\left(\left(_{\downarrow}\left(\bm{k}*\bm{l}\right)_{\downarrow}\right)^{\vee}\right)=0.
\]
\end{thm}

\section{Proof of the main results}

\subsection{Proof of Theorem \ref{thm:main}}

For the proof of the main results, it is convenient to introduce the
setting of Hoffman's algebra. Let $\mathbb{Q}\langle x,y\rangle$
be free non-commutative $\mathbb{Q}$-algebra generated by $x$ and
$y$. Then shuffle product $\shuffle:\mathbb{Q}\langle x,y\rangle\otimes\mathbb{Q}\langle x,y\rangle\to\mathbb{Q}\langle x,y\rangle$
is the $\mathbb{Q}$-linear map defined by the recursion
\begin{align*}
u\shuffle1 & =1\shuffle u=u,\qquad(u\in\mathbb{Q}\langle x,y\rangle)\\
(ua\shuffle vb) & =(u\shuffle vb)a+(ua\shuffle v)b\qquad(u,v\in\mathbb{Q}\langle x,y\rangle,\,a,b\in\{x,y\}).
\end{align*}
We define the $\mathbb{Q}$-linear map $I(0;-;1):\mathbb{Q}\langle x,y\rangle\to\mathbb{R}$
by
\[
I(0;yx^{k_{1}-1}\cdots yx^{k_{d}-1};1)=\zeta(k_{1},\dots,k_{d})\qquad(k_{d}\geq1),
\]
\[
I(0;x\shuffle u;1)=I(0;y\shuffle u;1)=0\qquad(u\in\mathbb{Q}\langle x,y\rangle).
\]
Then, $\zeta_{m}(k_{1},\dots,k_{d})$ can be written as
\begin{equation}
\zeta_{m}(k_{1},\dots,k_{d})=I(0;x^{m}yx^{k_{1}-1}\cdots yx^{k_{d}-1};1).\label{eq:shifted_integ}
\end{equation}
Furthermore define a $\mathbb{Q}$-linear map $I(1;-;0):\mathbb{Q}\langle x,y\rangle\to\mathbb{R}$
by
\begin{equation}
I(0;v_{1}\cdots v_{k};1)=(-1)^{k}I(1;v_{k}\cdots v_{1};0)\label{eq:integ_rev}
\end{equation}
where $v_{1},\dots,v_{k}\in\{x,y\}$. Then we have
\begin{equation}
I(0;v_{1}\cdots v_{k};1)=(-1)^{k}I(1;f_{x\leftrightarrow y}(v_{1}\cdots v_{k});0)\label{eq:integ_xy}
\end{equation}
where $f_{x\leftrightarrow y}$ is the automorphism of $\mathbb{Q}\langle x,y\rangle$
defined by $f_{x\leftrightarrow y}(x)=y$ and $f_{x\leftrightarrow y}(y)=x$.

\begin{proof}[Proof of Theorem \ref{thm:main}]
Recall that, by Theorem \ref{thm:Bou} (v), we have
\begin{equation}
\Psi_{k_{1},\dots,k_{d}}(z)=\sum_{j=1}^{d}\sum_{\substack{a+s+b=k_{j}-2\\
a,b,s\geq0
}
}(-1)^{k_{1}+\cdots+k_{j-1}+a}\zeta_{a}(k_{j-1},\dots,k_{1})\zeta_{b}(k_{j+1},\dots,k_{d})\Psi_{s+2}(z).\label{eq:e1}
\end{equation}
Here, we have
\begin{align*}
(-1)^{k_{1}+\cdots+k_{j-1}+a}\zeta_{a}(k_{j-1},\dots,k_{1}) & =(-1)^{k_{1}+\cdots+k_{j-1}+a}I(0;x^{a}yx^{k_{j-1}-1}\cdots yx^{k_{1}-1};1)\\
 & =I(1;x^{k_{1}-1}y\cdots x^{k_{j-1}-1}yx^{a};0)
\end{align*}
and
\[
\zeta_{b}(k_{j+1},\dots,k_{d})=I(0;x^{b}yx^{k_{j+1}-1}\cdots yx^{k_{d}};1)
\]
by (\ref{eq:shifted_integ}) and (\ref{eq:integ_rev}). Thus, (\ref{eq:e1})
can be rewritten as
\[
\Psi_{k_{1},\dots,k_{d}}(z)=\sum_{s\geq0}\sum_{\substack{w_{1}x^{s+1}w_{2}=\\
x^{k_{1}-1}yx^{k_{2}-1}\cdots yx^{k_{d}-1}
}
}I(1;w_{1};0)I(0;w_{2};1)\Psi_{s+2}(z)
\]
where $w_{1}$ and $w_{2}$ run all monomials in $\mathbb{Q}\langle x,y\rangle$
satisfying $w_{1}x^{s+1}w_{2}=x^{k_{1}-1}yx^{k_{2}-1}\cdots yx^{k_{d}-1}$.
Now, put $u_{j}:=f_{x\leftrightarrow y}(w_{j})$ for $j=1,2$. Then
\begin{align*}
f_{x\leftrightarrow y}(w_{1}x^{s+1}w_{2}) & =u_{1}y^{s+1}u_{2},\\
f_{x\leftrightarrow y}(x^{k_{1}-1}yx^{k_{2}-1}\cdots yx^{k_{d}-1}) & =yx^{l_{1}-1}\cdots yx^{l_{r}-1}y
\end{align*}
where $(l_{1},\dots,l_{r})=(_{\downarrow}\bm{k}_{\downarrow})^{\vee}$.
Thus,
\begin{align*}
\Psi_{k_{1},\dots,k_{d}}(z) & =\sum_{s\geq0}\sum_{\substack{u_{1}y^{s+1}u_{2}=\\
yx^{l_{1}-1}\cdots yx^{l_{r}-1}y
}
}I(1;f_{x\leftrightarrow y}(u_{1});0)I(0;f_{x\leftrightarrow y}(u_{2});1)\Psi_{s+2}(z)\\
 & =\sum_{s\geq0}(-1)^{l_{1}+\cdots+l_{r}-s}\sum_{\substack{u_{1}y^{s+1}u_{2}=\\
yx^{l_{1}-1}\cdots yx^{l_{r}-1}y
}
}I(0;u_{1};1)I(1;u_{2};0)\Psi_{s+2}(z)\qquad(\text{by (\ref{eq:integ_xy})})\\
 & =(-1)^{l_{1}+\cdots+l_{r}}\sum_{\substack{u_{1}yv=\\
yx^{l_{1}-1}\cdots yx^{l_{r}-1}y
}
}I(0;u_{1};1)\sum_{y^{s}u_{2}=v}(-1)^{s}I(1;u_{2};0)\Psi_{s+2}(z)\qquad(v\coloneqq y^{s}u_{2})\\
 & =(-1)^{l_{1}+\cdots+l_{r}}\sum_{j=0}^{r}I(0;yx^{l_{1}-1}\cdots yx^{l_{j}-1};1)\sum_{y^{s}u_{2}=x^{l_{j+1}-1}y\cdots x^{l_{r}-1}y}(-1)^{s}I(1;u_{2};0)\rho_{z}\left(\frac{T^{s}}{s!}\right).
\end{align*}
Here,
\[
I(0;yx^{l_{1}-1}\cdots yx^{l_{j}-1};1)=\zeta_{\shuffle}(l_{1},\dots,l_{j-1};0)
\]
and
\begin{align*}
\sum_{y^{s}u_{2}=x^{l_{j+1}-1}y\cdots x^{l_{r}-1}y}(-1)^{s}I(1;u_{2};0)\frac{T^{s}}{s!} & =(-1)^{l_{j+1}+\cdots+l_{r}}\sum_{wy^{s}=yx^{l_{r}-1}\cdots yx^{l_{j+1}-1}}I(0;w;1)\frac{T^{s}}{s!}\qquad(\text{by (\ref{eq:integ_rev})})\\
 & =(-1)^{l_{j+1}+\cdots+l_{r}}\zeta_{\shuffle}(l_{r},\dots,l_{j+1};T).
\end{align*}
Thus,
\begin{align*}
\Psi_{k_{1},\dots,k_{d}}(z) & =(-1)^{l_{1}+\cdots+l_{r}}\sum_{j=0}^{r}(-1)^{l_{j+1}+\cdots+l_{r}}\zeta_{\shuffle}(l_{1},\dots,l_{j-1};0)\rho_{z}\left(\zeta_{\shuffle}(l_{r},\dots,l_{j+1};T)\right)\\
 & =(-1)^{l_{1}+\cdots+l_{r}}\rho_{z}\left(\zeta_{\shuffle}^{S}(l_{1},\dots,l_{r};T)\right),
\end{align*}
which completes the proof since $(l_{1},\dots,l_{r})=(_{\downarrow}\bm{k}_{\downarrow})^{\vee}$
and $l_{1}+\cdots+l_{r}=k_{1}+\cdots+k_{d}-2$.
\end{proof}

\subsection{Proof of Theorem \ref{thm:main-conjBou}}
\begin{proof}[Proof of Theorem \ref{thm:main-conjBou} ]
Let $\bm{l}=(l_{1},\dots,l_{r})\in\mathbb{Z}_{\geq1}^{r}$ be any
non-empty index. Then, by Theorem \ref{thm:main}, we have
\[
(-1)^{l_{1}+\cdots+l_{r}}\rho_{z}(\zeta_{\shuffle}^{S}(\bm{l};T))=\Psi_{\bm{k}}(z)
\]
where $_{\downarrow}\bm{k}_{\downarrow}=\bm{l}^{\vee}$. Thus,
\[
(-1)^{l_{1}+\cdots+l_{r}}\zeta_{\shuffle}^{S}(\bm{l};0)\Psi_{2}(z)\equiv\Psi_{\bm{k}}(z)\pmod{\bigoplus_{s=1}^{l_{1}+\cdots+l_{r}}\mathsf{MZV}_{l_{1}+\cdots+l_{r}-s}\cdot\Psi_{s+2}(z)}.
\]
Therefore, if we put
\[
\mathsf{SMZV}_{w}=\mathrm{span}_{\mathbb{Q}}\langle\zeta_{\shuffle}^{S}(\bm{l};0)\mid\mathrm{wt}(\bm{l})=w\rangle,\quad\mathsf{SMZV}=\bigoplus_{w=0}^{\infty}\mathsf{SMZV}_{w},
\]
then
\[
\mathsf{SMZV}_{w}\cdot\Psi_{2}(z)\subset\mathsf{MTGF}_{w+2}+\bigoplus_{s=1}^{w}\mathsf{MZV}_{w-s}\cdot\Psi_{s+2}(z).
\]
In other words, we have
\begin{equation}
\mathsf{SMZV}\cdot\Psi_{2}(z)\subset\mathsf{MTGF}+\bigoplus_{s\geq3}\mathsf{MZV}\cdot\Psi_{s}(z)\label{eq:e_smzv}
\end{equation}
as the graded $\mathbb{Q}$-module where we let the degree of $\Psi_{k}$
be $k$. Furthermore, by Theorem \ref{thm:smzv-span-mzv}, we have
\[
\mathsf{MZV}=\mathsf{SMZV},
\]
and thus (\ref{eq:e_smzv}) is equivalent to
\begin{equation}
\mathsf{MZV}\cdot\Psi_{2}(z)\subset\mathsf{MTGF}+\bigoplus_{s\geq3}\mathsf{MZV}\cdot\Psi_{s}(z).\label{eq:e_mzv}
\end{equation}
By applying $(\frac{\partial}{\partial z})^{n-2}$ to (\ref{eq:e_mzv}),
since $\frac{\partial}{\partial z}(\mathsf{MTGF}_{k})\subset\mathsf{MTGF}_{k+1}$
by Theorem \ref{thm:Bou} (ii), we have
\[
\mathsf{MZV}\cdot\Psi_{n}(z)\subset\mathsf{MTGF}+\bigoplus_{s\geq n+1}\mathsf{MZV}\cdot\Psi_{s}(z)
\]
for $n\geq2$. Thus we have
\[
\mathsf{MZV}\cdot\Psi_{2}(z)\subset\mathsf{MTGF}+\bigoplus_{s\geq3}\mathsf{MZV}\cdot\Psi_{s}(z)\subset\mathsf{MTGF}+\bigoplus_{s\geq4}\mathsf{MZV}\cdot\Psi_{s}(z)\subset\cdots\subset\mathsf{MTGF},
\]
which implies Conjecture \ref{conj:Bou} by Theorem \ref{thm:Bou-equiv}
(c).
\end{proof}

\subsection{Proof of Theorem \ref{thm:main-Kawashima}}
\begin{lem}
\label{lem:psi_expand}For $s\geq2$, $\Psi_{s}(z)$ is a degree $s$
polynomial of $w=1/(e^{2\pi iz}-1)$ whose constant term vanishes
and the coefficient of $w^{1}$ is $(2\pi i)^{s}/(s-1)!$. In other
words, there exists $c_{2}^{(s)},\dots,c_{s}^{(s)}\in\mathbb{C}$
such that
\[
\Psi_{s}(z)=\frac{(2\pi i)^{s}}{(s-1)!}w+\sum_{k=2}^{s}c_{k}^{(s)}w^{k}.
\]
\end{lem}

\begin{proof}
We prove by induction on $s$. Note that 
\[
\frac{dw}{dz}=-2\pi i(w+w^{2}).
\]
The case $s=2$ follows from
\begin{align*}
\Psi_{2}(z) & =\sum_{m=-\infty}^{\infty}\frac{1}{(z+m)^{2}}=-\frac{d}{dz}\sum_{m=-\infty}^{\infty}\left(\frac{1}{z+m}-\frac{1}{m}\right)=-\frac{d}{dz}\pi\cot(\pi z)\\
 & =-\pi i\frac{d}{dz}\left(\frac{e^{2\pi iz}+1}{e^{2\pi iz}-1}\right)=-\pi i\frac{d}{dz}\left(1+2w\right)=(2\pi i)^{2}(w+w^{2}).
\end{align*}
Let $m\geq2$. Assume that the claim holds for $s=m$, i.e.,
\[
\Psi_{m}(z)=\frac{(2\pi i)^{m}}{(m-1)!}w+\sum_{k=2}^{m}c_{k}^{(m)}w^{k}.
\]
Then,
\begin{align*}
\Psi_{m+1}(z) & =-\frac{1}{m}\frac{d}{dz}\Psi_{m}(z)\\
 & =-\frac{1}{m}\frac{d}{dz}\left(\frac{(2\pi i)^{m}}{(m-1)!}w+\sum_{k=2}^{m}c_{k}^{(m)}w^{k}\right)\\
 & =\frac{(2\pi i)^{m+1}}{m!}(w+w^{2})+\frac{2\pi i}{m}\sum_{k=2}^{m}c_{k}^{(m)}k(w^{k}+w^{k+1}).
\end{align*}
Thus, the claim holds for $s=m+1$. Therefore, the claim holds for
all $s\geq2$ by induction.
\end{proof}
\begin{example}
The following is an example of the lemma for small $s$. Put $w=1/(e^{2\pi iz}-1)$.
Then
\begin{align*}
\Psi_{2}(z) & =(2\pi i)^{2}\left(w+w^{2}\right),\\
\Psi_{3}(z) & =\frac{(2\pi i)^{3}}{2}\left(w+3w^{2}+2w^{3}\right),\\
\Psi_{4}(z) & =\frac{(2\pi i)^{4}}{6}\left(w+7w^{2}+12w^{3}+6w^{4}\right).
\end{align*}
\end{example}

By Lemma \ref{lem:psi_expand}, $\rho_{z}(\zeta_{\shuffle}^{S}(\bm{l};T))$
is a polynomial of $w=1/(e^{2\pi iz}-1)$. The coefficient of $w^{1}$
is given by the following lemma. 

\begin{lem}
\label{lem:coeff_is_RS}Let $\bm{l}=(l_{1},\dots,l_{r})\in\mathbb{Z}_{\geq1}^{r}$.
Put $w=1/(e^{2\pi iz}-1)$. Then the coefficient of $w^{1}$ in $\rho_{z}(\zeta_{\shuffle}^{S}(\bm{l};T))$
is given by $(2\pi i)^{2}\zeta^{RS}(\bm{l})$, i.e.,
\[
\rho_{z}(\zeta_{\shuffle}^{S}(\bm{l};T))=(2\pi i)^{2}\zeta^{RS}(\bm{l})w+w^{2}P_{\bm{l}}(w)
\]
for some polynomial $P_{\bm{l}}(T)\in\mathbb{C}[T]$.
\end{lem}

\begin{proof}
Let $Q(T)=\sum_{k}c_{k}T^{k}\in\mathbb{C}[T]$ be any polynomial.
Then
\[
\rho_{z}(Q(T))=\sum_{k}c_{k}k!\Psi_{k+2}(z).
\]
Thus, by Lemma \ref{lem:psi_expand}, the coefficient of $w^{1}$
in $\rho_{z}(Q(T))$ is given by
\[
\sum_{k}c_{k}k!\frac{(2\pi i)^{k+2}}{(k+1)!}=\sum_{k}c_{k}\frac{(2\pi i)^{k+2}}{k+1}=2\pi i\sum_{k}c_{k}\int_{0}^{2\pi i}T^{k}dT=2\pi i\int_{0}^{2\pi i}Q(T)dT.
\]
Thus, by considering the case $Q(T)=\zeta_{\shuffle}^{S}(\bm{l};T)$,
the coefficient of $w^{1}$ in $\rho_{z}(\zeta_{\shuffle}^{S}(\bm{l};T))$
is given by
\[
2\pi i\int_{0}^{2\pi i}\zeta_{\shuffle}^{S}(\bm{l};T)dT,
\]
which equals to
\[
(2\pi i)^{2}\zeta^{RS}(\bm{l})
\]
by (\ref{eq:zeta_RS}).
\end{proof}
Now, we can prove the last main theorem.
\begin{proof}[Theorem \ref{thm:main-Kawashima}]
By Theorem \ref{thm:Bou} (iv) and Theorem \ref{thm:main}, we have
\begin{equation}
\rho_{z}\left(\zeta_{\shuffle}^{S}\left(\left(_{\downarrow}(\bm{k}*\bm{l})_{\downarrow}\right)^{\vee};T\right)\right)=\rho_{z}\left(\zeta_{\shuffle}^{S}\left(\left(_{\downarrow}\bm{k}_{\downarrow}\right)^{\vee};T\right)\right)\cdot\rho_{z}\left(\zeta_{\shuffle}^{S}\left(\left(_{\downarrow}\bm{l}{}_{\downarrow}\right)^{\vee};T\right)\right).\label{eq:zeta_S_prod}
\end{equation}
Then both sides are polynomial of $w=1/(e^{2\pi iz}-1)$. The coefficients
of $w^{1}$ in the left-hand side is 
\[
(2\pi i)^{2}\zeta^{RS}\left(\left(_{\downarrow}(\bm{k}*\bm{l})_{\downarrow}\right)^{\vee}\right)
\]
by Lemma \ref{lem:coeff_is_RS}, and the coefficient of $w^{1}$ in
the right-hand side is $0$ by Lemma \ref{lem:psi_expand}. Thus,
\[
\zeta^{RS}\left(\left(_{\downarrow}(\bm{k}*\bm{l})_{\downarrow}\right)^{\vee}\right)=0,
\]
which completes the proof.
\end{proof}
\begin{rem}
Theorem \ref{thm:main-Kawashima} can be regarded as an analogue of
the linear part of Kawashima's relation \cite[Corollary 5.5]{kawashima}
in the following sense. Let
\[
\hat{\zeta}^{\star}(k_{1},\dots,k_{r})=(-1)^{r}\zeta^{\star}(k_{1},\dots,k_{r})
\]
where $\zeta^{\star}(k_{1},\dots,k_{r})$ is the multiple zeta star
values $\sum_{0<m_{1}\leq\cdots\leq m_{r}}m_{1}^{-k_{1}}\cdots m_{r}^{-k_{r}}$.
For an index $\bm{k}=(k_{1},\dots,k_{r})\in\mathbb{Z}_{\geq1}^{r}$,
define $\bm{k}_{\uparrow}$ and the dual index $\bm{k}^{\dagger}$
(when $k_{r}\geq2$) by
\[
(k_{1},\dots,k_{r})_{\uparrow}=(k_{1},\dots,k_{r-1},k_{r}+1)
\]
\[
(\overbrace{1,\dots,1}^{a_{1}-1},b_{1}+1,\dots,\overbrace{1,\dots,1}^{a_{s}-1},b_{s}+1)^{\dagger}=(\overbrace{1,\dots,1}^{b_{s}-1},a_{s}+1,\dots,\overbrace{1,\dots,1}^{b_{1}-1},a_{1}+1)^{\dagger}.
\]
Then the linear part of Kawashima's relation and Theorem \ref{thm:main-Kawashima}
have similar expressions:
\begin{align*}
\hat{\zeta}^{\star}(((\bm{k}*\bm{l})_{\uparrow})^{\dagger}) & =0\qquad(\text{the linear part of Kawashima's relation}),\\
\zeta^{RS}\left(\left(_{\downarrow}(\bm{k}*\bm{l})_{\downarrow}\right)^{\vee}\right) & =0\qquad(\text{Theorem \ref{thm:main-Kawashima}}).
\end{align*}
Furthermore, $\hat{\zeta}^{\star}$ and $\zeta^{RS}$ have the following
similarities. 
\end{rem}

\begin{itemize}
\item They satisfy the harmonic product formula
\begin{align*}
\hat{\zeta}^{\star}(\bm{k}*\bm{l}) & =\hat{\zeta}^{\star}(\bm{k})\hat{\zeta}^{\star}(\bm{l}),\\
\zeta^{RS}(\bm{k}*\bm{l}) & =\zeta^{RS}(\bm{k})\zeta^{RS}(\bm{l}).
\end{align*}
\item They do not satisfy simple two-term duality relation, but, their star-sums
\begin{align*}
\sum_{\square\in\{\,'+',\,','\,\}}\hat{\zeta}^{\star}(k_{1}\square\cdots\square k_{r}) & =(-1)^{r}\zeta(k_{1},\dots,k_{r}),\\
\sum_{\square\in\{\,'+',\,','\,\}}\zeta^{RS}(k_{1}\square\cdots\square k_{r}) & =\zeta^{RS,\star}(k_{1},\dots,k_{r}),
\end{align*}
satisfy simple two-term duality relations
\begin{align*}
\zeta(\bm{k}^{\dagger}) & =\zeta(\bm{k}),\\
\zeta^{RS,\star}(\bm{k}^{\vee}) & =-\overline{\zeta^{RS,\star}(\bm{k})}.
\end{align*}
\end{itemize}
\begin{rem}
We can obtain a more general formula by taking the coefficient of
$w^{j}$ in (\ref{eq:zeta_S_prod}) for general $j\geq1$. It might
be an analogue of (a general form of) Kawashima's relation \cite[Corollary 5.4]{kawashima}.
\end{rem}

\subsection*{Acknowledgements}

This work was supported by JSPS KAKENHI Grant Numbers JP18K13392 and
JP22K03244.

\end{document}